\documentclass[11pt,letterpaper]{article}
\usepackage{amsfonts, amsmath, amssymb, amscd, amsthm, graphicx, mathrsfs, wasysym, setspace, mdwlist, color}
\usepackage{hyperref}
\swapnumbers

\theoremstyle{plain}
\newtheorem{Thm}{Theorem}[section]
\newtheorem{ThmA}[Thm]{Theorem A}
\newtheorem*{ThmA*}{Theorem A}

\newtheorem*{ThmB*}{Theorem B}
\newtheorem{Prop}[Thm]{Proposition}

\newtheorem{Lem}[Thm]{Lemma}

\theoremstyle{definition}
\newtheorem{Rem}[Thm]{Remark}

\newtheorem{Def}[Thm]{Definition}

\newcommand{\G}{\Gamma}
\newcommand{\T}{\mathcal{T}}
\newcommand{\rank}{\textrm{rank}}

\begin{document}

\title{Kurosh rank of intersections of subgroups of free products of right-orderable groups}

\author{Y. Antol\'{i}n, A. Martino and I. Schwabrow}

\date{\today}
\maketitle
\begin{abstract}
We prove that the reduced Kurosh rank of the intersection of two subgroups $H$ and $K$ of a free product of right-orderable groups is bounded above by the product of the reduced Kurosh ranks of $H$ and $K$.

In particular, taking the fundamental group of a graph of groups with trivial vertex and edge groups, and its Bass-Serre tree, our Theorem becomes the desired inequality of the usual Strengthened Hanna Neumann conjecture for free groups.
\medskip

{\footnotesize

}

\end{abstract}
\section{Introduction}
Let $H$ and $K$ be subgroups of a free group $F,$ and $\overline{r}(H)=\mathop{ \textrm{max}}\{0,\mathop{\textrm{rank}}(H)-1 \}.$ It was an open problem dating back to the 1950's to find an optimal bound of the rank of $H\cap K$ in terms of the ranks of $H$ and $K.$

In \cite{HNeumann}, Hanna Neumann proved  the following
\begin{equation}\label{eq:Hneumann}\overline{r}(H\cap K)\leq 2 \cdot \overline{r}(H)\cdot\overline{r}(K).
\end{equation}
 The \emph{Hanna Neumann Conjecture} says that \eqref{eq:Hneumann} holds replacing the 2 with a 1.

Later, in \cite{WNeumann}, Walter Neumann improved  \eqref{eq:Hneumann} to
\begin{equation}\label{eq:Wneumann}\sum_{g\in K\backslash F/H}\overline{r}(H^g\cap K)\leq 2\cdot  \overline{r}(H)\cdot\overline{r}(K),\end{equation} where $H^g=g^{-1} Hg.$ The \emph{ Strengthened Hanna Neumann Conjecture}, introduced by Walter Neumann, says that  \eqref{eq:Wneumann} holds replacing the 2 with a 1.

These two conjectures have received a lot of attention, and recently,  Igor Mineyev \cite{Mineyev1} proved that both conjectures are true\footnote{Independently and at the same time, Joel Friedman also proved these conjectures (See \cite{Dicks-F,Friedman})}. The fact that $F$ is right-orderable plays a crucial role in Mineyev's proof.

\begin{Def}\label{def:ord}
 A group $G$ is \emph{right-orderable} if it admits a total order which is invariant under the right $G$-multiplication action.
\end{Def}
In particular, right-orderability is inherited by subgroups and implies torsion freeness.

We recall the well known Kurosh Subgroup Theorem \cite{Kurosh}, whose proof can be found in, for example, \cite{Dicks-Dunwoody}.

\begin{Thm}[Kurosh Subgroup Theorem]
\label{kurosh}
Let $G=\ast_{i\in I} A_i$ be a free product and $H$ a subgroup of $G$. Then \begin{equation}\label{eq:kuroshdecom}
H= \ast (H \cap  A_i^g ) \ast F,
\end{equation} where the $g$ ranges over a set of double coset representatives in $A_i \backslash G / H$ for each $i\in I$ and $F$ is a free group.

\end{Thm}
In view of the theorem, one would like to define the Kurosh rank of $H$ with respect to the free product $\ast_{i\in I}A_i$ as the number of non-trivial factors $(H\cap  A_i^g )$ in \eqref{eq:kuroshdecom} plus the rank of $F$. One has to prove that this number is independent of the double coset representatives. This is done in  \cite[Lemma 3.4]{Ivanov08}.

However we prefer to give a different definition based on groups acting on trees.
In our approach, the Kurosh rank of $H\leqslant G$ will depend on the action on a $G$-tree $T$ rather in a free product decomposition of $G$. We will denote the Kurosh rank by $\kappa_T(H)$ and the reduced Kurosh rank  by $\overline{\kappa}_T(H)=\max\{0, \kappa_T(H)-1\}.$ We will give the formal definition of $\kappa_T$ in Section \ref{sec:kr}.

It is natural to consider to what extent Mineyev's Theorem can be generalised to intersections of subgroups of free products. It turns out that the Kurosh rank is the appropriate concept to consider in this context.  A proof of the Howson property for free products with respect to the Kurosh rank can be found in  \cite[Theorem 2.13 (1)]{Sykiotis}.

Let $G$ be a group, $T$ a $G$-tree with trivial edge stabilisers and $H,K\leqslant G$. Implicitly in \cite[Theorem 3]{Soma}, Soma proved that $$\overline{\kappa}_T( H\cap K ) \leq 18 \overline{\kappa}_T(H) \overline{\kappa}_T(K).$$ To our knowledge, Burns,  Chau and  Kam, were the first studying explicitly the kurosh rank of the intersection of subgroups in \cite{Burns-Chau-Kam}. %and they showed in$$\overline{\kappa}_T( H\cap K ) \leq 2 \overline{\kappa}_T(H) \overline{\kappa}_T(K) + 2 \min \{ \overline{\kappa}_T(H),\overline{\kappa}_T(K)\}.$$
Ivanov \cite{Ivanov01} found  that the inequality  \eqref{eq:Hneumann} holds for factor free subgroups of free products of right-orderable groups. Later, Dicks and Ivanov \cite{Dicks-I} extended this result to factor free subgroups of free products of torsion-free groups. In \cite{Ivanov08}, Ivanov using results obtained in \cite{Dicks-I} proved that if $G$ is torsion free then
$$\overline{\kappa}_T( H\cap K ) \leq 2 \overline{\kappa}_T(H) \overline{\kappa}_T(K).$$

Our main result is the following
\begin{ThmA}
Let $G$ be an right-orderable group and $T$ a $G$-tree with trivial edge stabilisers. Let $H,K$ be subgroups of $G.$  Then $$\sum_{g\in K \backslash G /H}\overline{\kappa}_T( H^g\cap K ) \leq \overline{\kappa}_T(H) \overline{\kappa}_T(K).$$
\end{ThmA}

\begin{Rem}
We note that if one considers a free group acting freely on a tree, then the Kurosh rank of a subgroup and its genuine rank agree. Therefore the Strengthened Hanna Neumann Conjecture is a corollary of Theorem A.
\end{Rem}

\begin{Rem}
The main interest is when all these quantities are finite, but it is also true when they are infinite, in which case we adopt the convention that $0 \cdot \infty=0$.
\end{Rem}

Mineyev's brilliant Hilbert-module proof \cite{Mineyev1}  yields to a very general result. Our proof of Theorem A, is basically the same as the simplified version of Mineyev's proof due to Warren Dicks \cite{Dicks}, which only applies to the special case of the strengthened Hanna Neumman conjecture.  A proof based on \cite{Mineyev1} and \cite{Dicks} that uses neither Hilbert-module theory nor
Bass-Serre theory was given in \cite{Mineyev2}.

We note that the main tool for passing from the free group case and the free product case is the proof of Theorem~\ref{thm:formula} and its application in Theorem~\ref{T:Main}, which says that the reduced Kurosh rank is equal to the number of orbits of edges in the Dicks tree.

To obtain examples of groups and trees for Theorem A, we can consider $G$ to be a graph of groups with right-orderable vertex groups and trivial edge groups,  and $T$ the corresponding Bass-Serre tree.  In this case, the group is a free product of right-orderable groups.  One can prove that such  group is right-orderable  using the Kurosh Subgroup Theorem and \cite[Theorem 2]{Burns-Hale}. Another fairly simple proof of this fact is given in \cite[Corollary 36]{Bergman}.

%%%%%%%%
% NEW SECTION
%%%%%%%%
\section{Kurosh rank}\label{sec:kr}
Our notation and basic reference for groups acting on trees is \cite{Dicks-Dunwoody}. The groups are acting on the right.

\begin{Def}[Kurosh rank]\label{def:kurosh}
Let $G$ be a group and $T$ a $G$-tree with trivial edge stabilisers and $H$ a subgroup of $G.$

Let $c_T(H) \in \mathbb{N} \cup \{ \infty \}$ be the number of vertices $vH\in VT/H$ such that $v$ has a non-trivial $H$-stabiliser. This is well defined since it is independent of the choice of the representative of $vH.$

The {\em Kurosh rank} of $H$ with respect to $T$ is defined to be
$$\kappa_T(H):=c_T(H)+\mathop{\textrm{rank}} (T/H),$$
where the $\mathop{\textrm{rank}}(T/H),$ is the number of edges outside of a maximal subtree of $T/H$ (equivalently, $\rank(T/H)$  is the rank of fundamental group of the graph $T/H$, which is a free group).

The {\em reduced Kurosh rank} of $H$ with respect to $T$ is defined to be
$$\overline{\kappa}_T(H):= \max\{0, \kappa_T(H)-1\}.$$
\end{Def}

\begin{Rem}
We note that the Kurosh rank of a subgroup depends on the free product decomposition and not just the isomorphism type of the subgroup. Therefore, suppose we took the free product of two surface groups, $S_1, S_2$. Then a free subgroup of infinite rank inside $S_1$ would have Kurosh rank 1, whereas a free subgroup of infinite rank meeting no conjugate of either $S_1$ or $S_2$ would have infinite Kurosh rank.
\end{Rem}

\begin{Prop}\label{P:kappa}
Let $H$ be a subgroup of $G$.
Let $T'$ be a $H$-subtree of $T$. Then $\kappa_T(H)=\kappa_{T'}(H).$

In particular, $T_H/H$ is finite (as a graph) if and only if $\kappa_T(H)$ is finite.
\end{Prop}

\begin{proof} This is an easy exercise in Bass-Serre theory.
\end{proof}

\begin{Thm}\label{thm:formula}
Let $T$ be and $H$-tree and let $\T$ an $H$-tree obtained from $T$ by collapsing an $H$-subforest of $T$. Then
\begin{equation}\label{eq:kurfor}
\kappa_T(H)=\sum_{v\in V\T/H} \kappa_T(H_v)+\rank(\T/H).
\end{equation}
In particular, the left hand side is infinite if and only if the right hand side is infinite.
\end{Thm}
\begin{proof}
Let $\phi\colon T\to \T$ be the natural quotient map, notice that this is an $H$-equivariant graph map. We let $\overline{\phi}$ be the induced graph map $T/H\to \T/H$.

As $\phi$ consists of the collapse of various subtrees to points, we have that $\phi^{-1}(v)=T_v$ is connected for any vertex $v$ of $T$. Hence $T_v$ is an $H_v$-subtree of $T$, where $H_v$ is the stabiliser of $v\in V\T$.

Note that if, for some $h \in H$ and $T_vh \cap T_v \neq \emptyset$, then $vh=v$ and hence $h \in H_v$. In particular, the graph map $T_v/H_v\to T/H$ is injective.

For every $vH\in V\T/H$ we let $\G_v$ be the subgraph $T_v/H_v$ of $T/H$ and  $(\G_v, H_v(-))$ be the associated graph of groups, which has fundamental group $H_v$ by Bass-Serre Theory. Since every vertex of $T$ is in the pre-image of some vertex of $\T$, we have that  $$VT/H=\bigsqcup_{vH \in V\T/H} V\G_v.$$

We view the edge set $E \T$ as a subset of $ET$, and the remaining edges of $T$ are precisely those which map to a vertex in $\T$. Hence it is clear that,
\begin{equation}
\label{flavour}
ET/H=\bigsqcup_{vH\in V\T/H} E\G_v \sqcup \bigsqcup_{eH\in E\T/H} eH.
\end{equation}

We now consider the rank of the graph, $T/H$, which is simply the number of edges outside of any maximal subtree. We construct a maximal subtree first by taking the union of maximal subtrees of each $\G_v$, and then enlarging the resulting forest to a maximal subtree, $Y$, of $T/H$. We note that the map $\overline{\phi}$ simply consists of collapsing each $\G_v$ to a point. Since each $Y \cap \G_v$ is connected, $\overline{\phi}(Y)$ must be a tree, and contains every vertex of $\T/H$ as $\overline{\phi}$ is surjective. Hence $\overline{\phi}(Y)$ is a maximal subtree of $\T/H$. As before, we think of edges of $\T/H$ as being a subset of the edges of $T/H$ and, hence the edges of $\overline{\phi}(Y)$ as being a subset of the edges of $Y$.

Therefore, counting edges outside of $Y$ and bearing in mind the disjoint union given by \eqref{flavour} immediately gives that,
\begin{align}\label{eq:ranks}
\rank(T/H) &= \sum_{vH\in V \T/H}\rank (\G_v) + \rank(\T/H).
\end{align}

In particular the left hand side is infinite if and only if (any term of) the right hand side is infinite.

By Proposition \ref{P:kappa}, $\kappa_T(H_v)=\kappa_{T_v}(H_v)$. Since $T_v/H_v=\G_v$, we have that $\kappa_T(H_v)=\rank(\G_v)+c_{T_v}(H_v)$. Moreover, since the $\cup_{v\in V\T/H} V\G_v=VT/H$ and $(\G_v, H(-))$ is the restriction of the graph of groups $(T/H, H(-))$ to the subgraph $\G_v$ we have that $c_T(H)=\sum_{v\in V\T/H} c_{T_v}(H_v)$.
Thus
\begin{equation}\label{eq:ct}
c_T(H)+ \sum_{vH\in V\T/H} \rank(\G_v)=\sum_{vH\in V\T/H} \kappa_{T}(H_v).
\end{equation}
Again, the left hand side is infinite if and only if the right hand side is infinite.

Now summing \eqref{eq:ct} and \eqref{eq:ranks} we get that,
$$
\kappa_T(H)+ \sum_{vH\in V\T/H} \rank(\G_v)=\sum_{vH\in V\T/H} \kappa_T(H_v)+\rank(\T/H)+ \sum_{vH\in V\T/H} \rank(\G_v).
$$

Hence we are done if $\sum_{vH\in \T/H} \rank(\G_v)$ is finite. However, in the case that $\sum_{vH\in \T/H} \rank(\G_v)$ is infinite, we deduce from \eqref{eq:ranks} that the left hand side of \eqref{eq:kurfor} is infinite, and from \eqref{eq:ct} that the right hand side of \eqref{eq:kurfor} is infinite.

\end{proof}

%%%%%%%%
% NEW SECTION
%%%%%%%%

\section{Main Argument}

Throughout this section $G$ will be an right-orderable group and $T$ a $G$-tree with trivial edge stabilisers. $G$ will act on $T$ on the right.

An element $g \in G$ is called {\em elliptic} if it fixes a point in $T$ and is called {\em hyperbolic} if it does not.

Given a hyperbolic $g \in G$, the {\em axis} of $g$, denoted $A_g$, consists of the subtree of points displaced by the minimal amount by $g$ (with respect to the path metric). This is always non-empty and homeomorphic to the real line.

Associated to any non-trivial subgroup (not necessarily finitely generated), $H \leqslant G$ there is a minimal $H$-invariant subtree $T_H$ of $T$. In general, this will be the union of the axes of hyperbolic elements of $H$ except when every element of $H$ is elliptic. In this case, a result of Serre says that any finite set of elements of $H$ have a common fixed point and therefore, as edge stabilisers are trivial, there will be a unique point for the whole of $H$, in which case $T_H$ will be the fixed vertex for $H$.

\begin{Rem}\label{rem:finiteG}
If, in Theorem A, either $\kappa_T(H)$ or $\kappa_T(K)$ is equal to infinity, then the theorem holds. So without loss of generality we can assume that $\kappa_T(H)$ and $\kappa_T(K)$ are finite. Hence $\kappa_T(\langle H\cup K\rangle)$ is also finite, and in view of Proposition \ref{P:kappa}, we can change $G$ to $\langle H\cup K \rangle,$ $T$ to $T_{\langle H\cup K\rangle}$ and hence we can assume that $T/G$ is finite.
\end{Rem}

Throughout the rest of the section $T/G$ will be a finite graph.

We fix an order $<$ of $G,$ and we use it to construct an order on the edges of $T.$ We first put any total order on $ET/G,$ which is a finite set, and we denote it  again by $<.$ Then, we order $(ET/G)\times G$ lexicographically, and use the natural bijection of this set with $ET$ to order it. That is $(eG,g)\leq (fG,h)$ if and only if $eG<fG$ or $eG=fG$ and $g\leq h.$
This ordering on the edges of $T$  is invariant under the action of $G$. We henceforth fix this ordering.

\begin{Def}[Dick's Trees]
Let $T'$ be a subtree of $T$.

An edge $e$ of $T'$  is  a  {\it $T'$-bridge}  if there is a reduced bi-infinite path in $T'$, containing $e$ and in which $e$ is the $<$-largest edge.

For any subgroup $H$ of $G$, we call an edge an {\it $H$-bridge} if it is a $T_H$-bridge. Note that $H$ acts freely on the set of $H$-bridges.

Note that if $T_0 \subseteq T_1$ are subtrees of $T$ then any $T_0$-bridge is a $T_1$-bridge. Hence, if $H \leqslant K $ are subgroups of $G$, then any $H$-bridge is also a $K$-bridge.

An {\it $H$-island} is a component of $T_H$ after all the $H$-bridges have been removed. Note that $H$ acts on the set of $H$-islands.

The Dicks $H$-tree, $\mathcal{T}_H,$  is the $H$-tree whose vertices are the $H$-islands and whose edges are the $H$-bridges. Note that all the edge stabilisers in $\mathcal{T}_H$ are trivial.
\end{Def}

\begin{Rem}
The purely combinatoric concept of a $T$-bridge is introduced by Dicks in \cite{Dicks} and corresponds to an order-essential edges  in Mineyev's terminology \cite[Definition 2]{Mineyev1}.  The concept of islands correspond to relative components to the set of order-essential edges. The relationship between bridges and order-essential edges is not at all obvious from the definitions and that is why we use Dicks terminology.
\end{Rem}

%\begin{dfn}
%\label{euler}
%Let $H$ be a subgroup of $G$ acting on a tree $S$ with corresponding graph of groups $\mathcal{H}$. We define the Kurosh Euler Characteristic of this graph of groups to be $\chi^S_{\kappa}(H):=\sum_v (1 - \kappa(H_v)) - \sum_e (1 - \kappa(H_e))$ where the first sum is taken over the vertices and vertex groups $H_v$ and the second sum is taken over the edges and edge groups $H_e$.
%
%\medskip
%
%Note that if $H$ acts freely on the edges of $S$ then $\kappa(H_e)=0$ for all $e$. Moreover, if $\kappa(H_v)=1$ for all $v$, then $-\chi^S_{\kappa}(H)$ is equal to the number of edges in $S/H$.
%\end{dfn}

\begin{Prop}
\label{ecalc}
Let $1 \neq H$ be a subgroup of $G$ with $\kappa_T(H)<\infty$. Suppose that for every $H$-island $I$ in $T_H,$ the $H$-stabiliser $H_I$ has $\kappa_T(H_I)=1$. %Equivalently, suppose that every vertex stabiliser in $\mathcal{T}_H$ has Kurosh rank one.
Then  $\overline{\kappa}_T(H)$ is equal to $|E\mathcal{T}_H / H|,$ the number of orbits of edges in Dicks $H$-tree.
\end{Prop}

\begin{proof} As $\kappa_T(H) < \infty$, $T_H/H$ is finite, and hence so is $\T_H/H$.

Since $\kappa_T(H_I)=1$ for all $H$-island $I$, we have that $$\sum_{I\in V\T_H/H} \kappa_T(H_I)=|V\T_H/H|.$$
By  Theorem \ref{thm:formula} we have that
$$\kappa_T(H)=|E\T_H/H|-|V\T_H/H|+1+\sum_{I\in V\T_H/H} \kappa_T(H_I)=|E\T_H/H|+1. $$
Since $H\neq 1$, $\overline{\kappa}_T(H)= \kappa_T(H)-1=|E\T_H/H|.$
\end{proof}

\medskip

Therefore the goal will be to show that for the tree $\mathcal{T}_H$, the stabiliser of any vertex has Kurosh rank 1 (to obtain our main result  Theorem~\ref{T:Main}). Equivalently, we need to show that stabilisers of $H$-islands have Kurosh rank less than 2 (Proposition~\ref{two}) and are non-trivial (Proposition~\ref{notriv}).

\begin{Prop}
\label{two}
Suppose that $H \leqslant G$ and that $\kappa_T(H) \geq 2$. Then there is a $H$-bridge in $T_H$.
\end{Prop}
\begin{proof} %As before, $T_H$ is the minimal H-invariant subtree of $T$.
If $H$ fixes a vertex, $T_H$ will simply be this fixed vertex and $\kappa(H) \leq 1$, contradicting the hypothesis.
%Hence $T_H = \cup_{g\in H}A_g$  where the union is taken over the hyperbolic elements of $H$ and each axis is homeomorphic to the real line.

Now suppose that $T_H$ is a single line. If some $h\in H$ fixes a vertex, then $h^2$ fixes $T_H$ and hence $h^2=1$. This is impossible since $H$ is right-orderable, and hence torsion free. Thus $H$  acts freely on $T_H,$ i.e. $H\cong \mathbb{Z}$, so has Kurosh rank 1, again contradicting the hypothesis.

Hence there exist two hyperbolic elements $g,h\in H$ with distinct axes, $A_g \neq A_h$. These two might intersect non-trivially, but they can only intersect in finitely many edges. If they did intersect in an infinite ray, then the commutator $g h g^{-1} h^{-1}$ would fix an infinite subray and hence an edge. However, the action is free on the edge set, and hence the commutator would be the trivial element, implying that $g$ and $h$ commute and therefore have the same axis, contradicting our choice of $g$ and $h$.

Now choose a vertex $v\in A_g$ and let $p$ denote a path from $v g^{-1}$ to $vg$ with $<$-largest edge $e$. By replacing $g$ with $g^{-1}$ if necessary, we may assume that $e>eg.$ It follows that $e>eg>eg^2 >\ldots >e g^n \ldots$. Therefore $e$ is the largest edge in the infinite ray starting with $e$ and continuing in the positive direction of the axis. (Note that the action of an hyperbolic element $g$ on its axis $A_g$ induces an orientation of $ A_g$ with respect to which g translates in the positive direction.) Likewise $e g^n$ is the largest edge in the infinite ray starting with $e g^n$.

It follows that every infinite ray $p_{\infty}$ in $A_g$ starting at any vertex of $A_g$ and going in the positive direction has a $<$-largest edge. Similarly,  every infinite ray $r_{\infty}$ in $A_h$ starting at any vertex of $A_h$ and going in the positive direction has a $<$-largest edge.

Thus there will exist a reduced bi-infinite path of the form $p^{-1}_{\infty} \cdot q\cdot r_{\infty}$ where $q$ is a finite path from $A_g$ to $A_h$; in the case where $A_g$ and $A_h$ are disjoint, $q$ is the path from one axis to the other, and in the case where they intersect, $q$ is a subpath of the intersection, possibly a single vertex. In either case, $p^{-1}_{\infty} \cdot q\cdot r_{\infty}$ has a $<$-largest edge which is then a $H$-bridge in $T_H$.
\end{proof}

\begin{Lem}\label{lem:inf}
Suppose that $H \leqslant G$ and that $\kappa_T(H)=\infty$. Then there are infinitely many $H$-orbits  of $H$-bridges in $\T_H$.
\end{Lem}
\begin{proof}
If $|E\T_H/H|<\infty$, then $|V\T_H/H|<\infty$ and also $\rank(\T/H) < \infty$.  Hence, by Theorem \ref{thm:formula}, $\sum_{vH\in V\T/H} \kappa_{T}(H_v)=\infty$. Therefore there exists an $vH\in V\T/H$ such that $\kappa_T(H_v)=\infty$. So by Proposition \ref{two}, $T_{H_v}$ contains an $H_v$-bridge, and hence an $H$-bridge. This contradicts the fact that  $H_v$ is the stabiliser of an island.

\end{proof}

\begin{Prop}
\label{notriv}
Let $H$ be a non-trivial subgroup of $G$  with $\kappa_T(H)<\infty$ and let $I$ be an $H$-island in $T_H$ with stabiliser $H_I$. Then $H_I$ is non-trivial.
\end{Prop}
\begin{proof}

If there is no $H$-bridge, then $H_I=H$ and, since $H$ is non-trivial, $H_I$ is non-trivial.

Consider the set $\{ e \ : \ e  \textrm{ is a bridge whose initial point is in } I \} $. Note that $T_H/H$ is a finite graph, therefore if the set above is infinite, it must contain edges $e$ and $eh$ for some $1 \neq h \in H$. Clearly, by looking at initial points, $h$ is a non-trivial element preserving $I$ and we would have that $H_I$ would be non-trivial.

Therefore, we may assume that the set above is finite and list the elements in order, $e_1 < e_2 < \ldots < e_s$.

Then $e_1$ is the largest edge in a reduced bi-infinite path in $T_H$. The tail of this path is a ray whose initial vertex is in $I$ and hence the entire ray must remain in $I$ due to the minimality of $e_1$. Therefore, again as $T_H/H$ is finite, there are two distinct edges of $I$ in the same $H$-orbit, and therefore $H_I$ is non-trivial (and note that this is really a contradiction, since it implies that the set above is infinite unless $\mathcal{T}_H$ is a single vertex).

\end{proof}

We now summarise the three propositions in one theorem.
\begin{Thm}\label{T:Main}
Let $G$ be an right-orderable group and $T$ an $G$-tree with trivial edge stabilisers and finitely many orbits of edges. Let $H$ be a subgroup of $G.$ Then $\overline{\kappa}_T(H)$ is equal to $|E\mathcal{T}_H / H|,$ the number of orbits of edges in Dicks $H$-tree.
\end{Thm}
\begin{proof}
If $H$ is the trivial group, then the theorem holds. So we assume $H$ is non-trivial.
If $\kappa_T(H)=\infty,$ then by Lemma \ref{lem:inf}, $\mathcal{T}_H/H$ is also infinite, and the theorem holds. So we assume $\kappa_T(H)<\infty.$
For every  $H$-island $I,$ by Proposition~\ref{notriv}, its $H$-stabiliser $H_I$ is non-trivial.
By Proposition~\ref{two}, if $\kappa(H_I) \geq 2$ then there would be a $T_{H_I}$-bridge, which would imply the existence of an $H$-bridge in $I$, contradicting the definition of $I$. Then $\kappa(H_I)=1.$
The theorem now follows from  Proposition~\ref{ecalc}.
\end{proof}

\begin{proof}[Proof of Theorem A]
If either $H$ or $K$ is trivial, the theorem trivially holds. Hence, we assume that $H$ and $K$ are non-trivial.

By Remark \ref{rem:finiteG}, the theorem also hold when the Kurosh rank of $H$ or $K$ is infinite. Therefore, again by Remark \ref{rem:finiteG}, we can assume that $T/G,$ $T_H/H$ and $T_K/K$ are finite.

The inclusions $T_{(H^g \cap K)} \subseteq T_{H^g}=(T_H)g$ and $T_{(H^g \cap K)} \subseteq T_K$ induce a (diagonal) graph map, sending $e(H^g \cap K)$ to $(e g^{-1} H , eK)$. Therefore we get a map,

$$
 \bigcup_{g\in K \backslash G /H} T_{(H^g \cap K)} / (H^g \cap K) \to T_H / H \times T_K / K
$$ which is injective on edges as edge stabilisers are trivial and the union is over distinct double coset representatives.

In turn, this induces a map on the following edge sets:
\begin{equation}\label{eq:map}
\bigcup_{g\in K \backslash G /H}E\mathcal{T}_{(H^g \cap K)} / (H^g \cap K) \to E\mathcal{T}_H / H \times E\mathcal{T}_K / K
\end{equation} which is also injective.

By Theorem \ref{T:Main}  we know that the number of edges in $\mathcal{T}_H/H$ is equal to $\overline{\kappa}_T(H)$. And similarly for $K$ and $H^g \cap K$ whenever they are non-trivial.

Therefore, the injectivity of the map \eqref{eq:map} on edges, gives us the result.
\end{proof}

\begin{Rem}
We note that in general, there is no obvious way to extend the map~\eqref{eq:map} to a graph map. That is because, while $(H \cap K)$-bridges map to $H$-bridges, it does not follow that non-bridges map to non-bridges. In general, an $(H \cap K)$-island will consist of multiple $H$ islands as well as some $H$-bridges.

For instance, if $H$ is free of rank 2 and $K$ is any cyclic subgroup. Then $T_{H \cap K}=T_K$ is a line and a $K$-island. However, this line will contain $H$-bridges unless the generator of $K$ acts elliptically on $\mathcal{T}_H$, which is not always the case.
\end{Rem}

\bigskip

\noindent{\textbf{\Large{Acknowledgments}}}
The first author was supported by MCI (Spain) through
project MTM2008-01550 and EPSRC through project EP/H032428/1.
The third named was supported by the Nuffield Foundation Grant URB/ 39357 during the period of this research.

\medskip
{\footnotesize }

\bibliographystyle{amsplain}

\bigskip

\textsc{Yago Antol\'{i}n, School of Mathematics,
University of  Southampton, University Road,
Southampton SO17 1BJ, UK}

\emph{E-mail address}{:\;\;}\url{yago.anpi@gmail.com}

\emph{URL}{:\;\;}\url{http://sites.google.com/site/yagoanpi/}

\medskip

\textsc{Armando Martino, School of Mathematics,
University of  Southampton, University Road,
Southampton SO17 1BJ, UK}

\emph{E-mail address}{:\;\;}\url{A.Martino@soton.ac.uk}

\emph{URL}{:\;\;}\url{http://www.personal.soton.ac.uk/am1t07/}

\medskip

\textsc{Inga Schwabrow, School of Mathematics,
University of  Southampton, University Road,
Southampton SO17 1BJ, UK}

\emph{E-mail address}{:\;\;}\url{is3g08@soton.ac.uk}

\end{document}